\documentclass[11pt,a4paper]{article}
\usepackage[english]{babel}
\usepackage{amsmath,amsfonts,amsthm,amssymb,amsbsy,upref,color,graphicx,amscd,hyperref,makeidx,enumerate,bbm,comment}
\usepackage[a4paper,margin=2.73truecm]{geometry}
\usepackage[active]{srcltx}
\usepackage[latin1]{inputenc}

\DeclareMathOperator{\dive}{div}

\def\ds{\displaystyle}
\def\eps{{\varepsilon}}

\def\O{\Omega}
\def\R{\mathbb{R}}

\def\A{\mathcal{A}}

\def\HH{\mathcal{H}}

\newcommand{\be}{\begin{equation}}
\newcommand{\ee}{\end{equation}}
\newcommand{\bib}[4]{\bibitem{#1}{\sc#2: }{\it#3. }{#4.}}

\numberwithin{equation}{section}
\theoremstyle{plain}
\newtheorem{teo}{Theorem}[section]

\newtheorem{prop}[teo]{Proposition}

\theoremstyle{remark}
\newtheorem{oss}[teo]{Remark}
\newtheorem{exam}[teo]{Example}
\newenvironment{ack}{{\bf Acknowledgements. }}

\title{Symmetry breaking for a problem in optimal insulation}

\author{Dorin Bucur, Giuseppe Buttazzo, Carlo Nitsch}

\begin{document}

\maketitle

\begin{abstract}
We consider the problem of optimally insulating a given domain $\O$ of $\R^d$; this amounts to solve a nonlinear variational problem, where the optimal thickness of the insulator is obtained as the boundary trace of the solution. We deal with two different criteria of optimization: the first one consists in the minimization of the total energy of the system, while the second one involves the first eigenvalue of the related differential operator. Surprisingly, the second optimization problem presents a symmetry breaking in the sense that for a ball the optimal thickness is nonsymmetric when the total amount of insulator is small enough. In the last section we discuss the shape optimization problem which is obtained letting $\O$ to vary too.
\end{abstract}



%%%%%%%%%%%%%%%%%%%%%%%%%%%%%%
\section{Introduction}\label{sintro}

In the present paper we deal with the problem of determining the best distribution of a given amount of insulating material around a fixed domain $\O$ of $\R^d$ which represents a thermally conducting body; the thickness of the insulating material is assumed very small with respect to the size of $\O$, so the material density is assumed to be a nonnegative function defined on the boundary $\partial\O$. A rigorous approach is to consider a limit problem when the thickness of the insulating layer goes to zero and simultaneously the conductivity in the layer goes to zero; this has been studied in \cite{brca80}, \cite{acbu86}, where the family of functionals
$$F_\eps(u)=\frac12\int_\O|\nabla u|^2\,dx+\frac\eps2\int_{\Sigma_\eps}|\nabla u|^2\,dx
-\int_\O fu\,dx$$
is considered on the Sobolev space $H^1_0(\O\cup\Sigma_\eps)$, being $\Sigma_\eps$ a thin layer of variable thickness $\eps h(\sigma)$ around the boundary $\partial\O$
$$\Sigma_\eps=\big\{\sigma+t\nu(\sigma)\ :
\ \sigma\in\partial\O,\ 0\le t<\eps h(\sigma)\big\}.$$
The temperature $u$ of the conducting body $\O$, with heat sources $f\in L^2(\O)$ and insulating distribution $h$, is then given by the minimization of the functional $F_\eps$ on $H^1_0(\O\cup\Sigma_\eps)$ or equivalently by the solution of the PDE
$$\begin{cases}
-\Delta u=f&\hbox{in }\O\\
-\Delta u=0&\hbox{in }\Sigma_\eps\\
u=0&\hbox{on }\partial(\O\cup\Sigma_\eps)\\
\ds\frac{\partial u^-}{\partial\nu}
=\eps\frac{\partial u^+}{\partial\nu}&\hbox{on }\partial\O.
\end{cases}$$
The last equality, represents a transmission condition across the boundary $\partial\O$, where $u^-$ and $u^+$ respectively denote the traces of $u$ in $\O$ and in $\Sigma_\eps$. Passing to the limit as $\eps\to0$ (in the sense of $\Gamma$-convergence) in the sequences of energy functionals provides the limit energy, given by
$$E(u,h)=\frac12\int_\O|\nabla u|^2\,dx+\frac12\int_{\partial\O}\frac{u^2}{h}\,d\HH^{N-1}-\int_\O fu\,dx.$$
Therefore the temperature $u$ solves the minimum problem
\be\label{energy}
E(h)=\min\big\{E(u,h)\ :\ u\in H^1(\O)\big\}
\ee
or equivalently the PDE
\be\label{pde}
\begin{cases}
-\Delta u=f&\hbox{in }\O\\
\ds h\frac{\partial u}{\partial\nu}+u=0&\hbox{on }\partial\O.
\end{cases}
\ee
Note that the boundary condition on $\partial\O$ is not any more the Dirichlet one, but is of a Robin type.

We are interested in determining the density $h(\sigma)$ which provides the best insulating performances, once the total amount of insulator is fixed, that is we consider density functions $h$ in the class
$$\HH_m=\left\{h:\partial\O\to\R\hbox{ measurable, }h\ge0,\ \int_{\partial\O}h\,d\HH^{d-1}=m\right\}.$$
The following two different optimization problems then arise.
\begin{itemize}
\item In a first problem the heat sources $f$ are given, and so the optimization problem we consider is written as
\be\label{pb1}
\min\big\{E(h)\ :\ h\in\HH_m\big\}.
\ee
where $E(h)$ is given by \eqref{energy}. Note that the energy $E(h)$ can be written in terms of the solution $u_h$ of the PDE \eqref{pde}; indeed, multiplying both sides of \eqref{pde} by $u_h$ and integrating by parts gives
$$E(h)=-\frac12\int_\O fu_h\,dx.$$
Thus, the minimization of $E(h)$ in \eqref{pb1} corresponds to the choice of $h$ which maximizes the quantity $\int_\O fu_h\,dx$. In particular, when $f=1$, this first optimization problem consists in placing $h$ around $\O$ in the best way to obtain the maximal average temperature in $\O$.

\item The second optimization problem we consider deals with the operator $\A$ written in a weak form as
\be\label{operator}
\langle\A u,\phi\rangle=\int_\O\nabla u\nabla\phi\,dx+\int_{\partial\O}\frac{u\phi}{h}\,d\HH^{d-1}
\ee
and the corresponding heat equation
$$\partial_tu+\A u=0,\qquad u(0,x)=u_0(x).$$
In this case the long time behavior of the temperature $u(t,x)$ is governed by the first eigenvalue $\lambda(h)$ of the operator $\A$, given by the Rayleigh quotient
$$\lambda(h)=\inf\left\{\frac{\int_\O|\nabla u|^2\,dx+\int_{\partial\O}h^{-1}u^2\,d\HH^{d-1}}{\int_\O u^2\,dx}\ :\ u\in H^1(\O),\ u\ne0\right\}.$$
The best insulation is in this case determined by the minimum problem
\be\label{pb2}
\min\big\{\lambda(h)\ :\ h\in\HH_m\big\}.
\ee
\end{itemize}

We will show that both problems \eqref{pb1} and \eqref{pb2} admit an optimal solution that can be recovered by solving some suitable auxiliary variational problems. In the energy case of problem \eqref{pb1} the optimal density $h_{opt}$ is unique and so, when $\O$ is a ball and $f$ radial, $h_{opt}$ is constant. On the contrary, for problem \eqref{pb2} we will see that a surprising symmetry breaking occurs: when $\O$ is a ball and $m$ is small enough the optimal density $h_{opt}$ is nonconstant.

In the last section of the paper we consider the shape optimization problems that arise from \eqref{pb1} and \eqref{pb2} when we let $\O$ to vary too, among domains having a prescribed volume.

%Even if it is not the aim of the present paper, in the last section we will point out the two shape optimization problems that arise from \eqref{pb1} and \eqref{pb2} when we let $\O$ to vary too, among domains having a prescribed volume.

%In both cases the existence of an optimal domain $\O_{opt}$ does not follow from any known result and probably requires a rather delicate analysis.

%%%%%%%%%%%%%%%%%%%%%%%%%%%%%%
\section{The energy problem}\label{senergy}

In this section we consider the optimization problem \eqref{pb1}; this problem was already considered in \cite{bu88} (see also \cite{bubu05}), for the sake of completeness we summarize the main results. The optimization problem we deal with is
$$\min_{h\in\HH_m}\ \min_{u\in H^1(\O)}\left\{\frac12\int_\O|\nabla u|^2\,dx+\frac12\int_{\partial\O}\frac{u^2}{h}\,d\HH^{N-1}-\int_\O fu\,dx\right\}$$
which, interchanging the two minimizations, gives
$$\min_{u\in H^1(\O)}\ \min_{h\in\HH_m}\left\{\frac12\int_\O|\nabla u|^2\,dx+\frac12\int_{\partial\O}\frac{u^2}{h}\,d\HH^{N-1}-\int_\O fu\,dx\right\}.$$
The minimum with respect to $h$ is easy to compute explicitly and, for a fixed $u\in H^1(\O)$, which does not identically vanish on $\partial\O$, is reached for
$$h=m\frac{|u|}{\int_{\partial\O}|u|\,d\HH^{d-1}}\;;$$
the choice of $h$ is irrelevant when $u\in H^1_0(\O)$.
Therefore, the optimization problem \eqref{pb1} can be rewritten as
\be\label{pb1aux}
\min\left\{\frac12\int_\O|\nabla u|^2\,dx+\frac{1}{2m}\Big(\int_{\partial\O}|u|\,d\HH^{d-1}\Big)^2-\int_\O fu\,dx\ :\ u\in H^1(\O)\right\}.
\ee
The existence of a solution for problem \eqref{pb1aux} follows by the Poincar\'e-type inequality (see Proposition 1.5.3 of \cite{bubu05})
$$\int_\O u^2\,dx\le C\left[\int_\O|\nabla u|^2\,dx+\Big(\int_{\partial\O}|u|\,d\HH^{d-1}\Big)^2\right]$$
which implies the coercivity of the functional in \eqref{pb1aux}. The solution is also unique, thanks to the result below.

\begin{prop}\label{unique}
Assume $\O$ is connected. Then, the functional
$$u\mapsto F(u)=\frac12\int_\O|\nabla u|^2\,dx+\frac{1}{2m}\Big(\int_{\partial\O}|u|\,d\HH^{d-1}\Big)^2$$
is strictly convex on $H^1(\O)$, hence for every $f\in L^2(\O)$ the minimization problem \eqref{pb1aux} admits a unique solution.
\end{prop}

\begin{proof}
%See Proposition 1.5.4 of \cite{bubu05}.
Let $u_1,u_2\in H^1(\O)$; since the term
$$\frac{1}{2m}\Big(\int_{\partial\O}|u|\,d\HH^{d-1}\Big)^2$$
is convex and
$$\int_\O\Big|\frac{\nabla u_1+\nabla u_2}{2}\Big|^2\,dx
=\int_\O\frac{|\nabla u_1|^2+|\nabla u_2|^2}{2}\,dx
-\int_\O\Big|\frac{\nabla u_1-\nabla u_2}{2}\Big|^2\,dx\;,$$
we have

$$F\Big(\frac{u_1+u_2}{2}\Big)<\frac{F(u_1)+F(u_2)}{2}$$
whenever $u_1-u_2$ is nonconstant. It remains to consider the case $u_1-u_2=c$ with $c$ constant.
If $u_1$ and $u_2$ have a different sign on a subset $B$ of $\partial\O$ with $\HH^{d-1}(B)>0$, we have
$$|u_1+u_2|<|u_1|+|u_2|\qquad\hbox{$\HH^{d-1}$-a.e. on }B\;,$$
which again gives the strict convexity of the functional $F$. Finally, if $u_1$ and $u_2$ have the same sign on $\partial\O$, we have
\[\begin{split}
&\Big(\int_{\partial\O}|u_1+u_2|\,d\HH^{d-1}\Big)^2
-2\Big(\int_{\partial\O}|u_1|\,d\HH^{d-1}\Big)^2
-2\Big(\int_{\partial\O}|u_2|\,d\HH^{d-1}\Big)^2\\
&\qquad=-\Big(\int_{\partial\O}|u_1-u_2|\,d\HH^{d-1}\Big)^2
=-c^2\HH^{d-1}(\partial\O)
\end{split}\]
which gives again the strict convexity of $F$ and concludes the proof.
\end{proof}

\begin{exam}\label{oneball}
Let $\O=B_R$ be the ball of radius $R$ in $\R^d$ and let $f=1$; then, by the uniqueness result of Proposition \ref{unique} the optimal solution $u$ of the minimization problem \eqref{pb1aux} is radially symmetric and is given by
$$u(r)=\frac{R^2-r^2}{2d}+c$$
for a suitable nonnegative constant $c$. The value of $c$ can be easily computed; indeed the energy of the function $u$ above is
\[\begin{split}
&\frac{d\omega_d}{2}\int_0^R r^{d-1}\Big(\frac{r}{d}\Big)^2\,dr
+\frac{1}{2m}(cd\omega_d R^{d-1})^2
-d\omega_d\int_0^R r^{d-1}\Big(\frac{R^2-r^2}{2d}+c\Big)\,dr\\
&=-\frac{\omega_d}{2d(d+2)}R^{d+2}+c^2\frac{d^2\omega_d^2}{2m}R^{2d-2}-c\omega_d R^d
\end{split}\]
where $\omega_d$ denotes the Lebesgue measure of the unit ball in $\R^d$. Optimizing with respect to $c$ we obtain
$$c_{opt}=\frac{m}{d^2\omega_d R^{d-2}}\;,$$
with minimal energy
$$-\frac{R^2}{2d}\Big(\frac{\omega_d R^d}{d+2}+\frac{m}{d}\Big)\;.$$
\end{exam}

\begin{oss}
If in Propostion \ref{unique} the domain $\O$ is not connected , then the strict convexity of the functional $F$ does not occur anymore, so that uniqueness in the minimization problem \eqref{pb1aux} does not hold. Nevertheless, the non-uniqueness issue is a matter of constants. Indeed, assume for simplicity that $\O$ has two connected components $\O_1,\O_2$ and $u,v \ge 0$ are such that
for some $t \in (0,1)$, $tF(u)+(1-t) F(v)= F(tu+(1-t)v)$. Then, one gets
$$\int_\O|\nabla u|^2\,dx= \int_\O|\nabla v|^2\,dx \;\;\mbox{ and } \;\; \int_{\partial\O}u\,d\HH^{d-1}= \int_{\partial\O}v\,d\HH^{d-1}.$$
This implies that $u\big|_{\O_1}-v\big|_{\O_1}=c_1$ and $u\big|_{\O_2}-v\big|_{\O_2}=c_2$ with
$$c_1\HH^{d-1}(\partial\O_1)+c_2\HH^{d-1}(\partial\O_2)=0\;.$$
\end{oss}

\begin{exam}\label{twoballs}
Let $\O=B_{R_1}\cup B_{R_2}$ be the domain of $\R^d$ made by the union of two disjoint balls of radius $R_1$ and $R_2$ respectively, and let $f=1$ as in Example \ref{oneball}. On each ball $B_{R_j}$ ($j=1,2$) the optimal solution $u$ of the minimization problem \eqref{pb1aux} is radially symmetric and is given by
$$u(r)=\frac{R_j^2-r^2}{2d}+c_j\qquad j=1,2$$
for suitable values of the constants $c_1,c_2\ge0$. Repeating the calculations made in Example \ref{oneball} we have that the energy of $u$ is given by
$$-\frac{\omega_d}{2d(d+2)}\big(R_1^{d+2}+R_2^{d+2}\big)+\frac{d^2\omega_d^2}{2m}\big(c_1R_1^{d-1}+c_2R_2^{d-1}\big)^2-\omega_d\big(c_1R_1^d+c_2R_2^d\big)\;.$$
Therefore, optimizing with respect to $c_1$ and $c_2$, we obtain easily that:
\begin{itemize}
\item if $R_1=R_2=R$ any choice of $c_1$ and $c_2$ with
$$c_1+c_2=\frac{m}{d^2\omega_d R^{d-2}}$$
is optimal, and provides the minimal energy
$$-\frac{R^2}{2d}\Big(\frac{\omega_d R^d}{d+2}+\frac{m}{d}\Big)\;.$$
\item if $R_1\ne R_2$ then necessarily one between $c_1$ and $c_2$ has to vanish and, if $R_1<R_2$, the optimal choice is
$$c_1=0\;,\qquad c_2=\frac{m}{d^2\omega_d R_2^{d-2}}\;.$$
In other words, if $R_1\ne R_2$ it is more efficient to concentrate all the insulator around the largest ball, leaving the smallest one unprotected.
\end{itemize}
\end{exam}

\begin{oss}
We notice that the Euler-Lagrange equation of the variational problem \eqref{pb1aux} is
\[\begin{cases}
-\Delta u=f\quad\mbox{in }\O,\\
\ds\frac{\partial u}{\partial\nu}=-\frac1m\int_{\partial\O}u\,d\sigma\quad\mbox{on }\partial\O\cap\{u>0\},\\
\ds\frac{\partial u}{\partial\nu}\ge-\frac1m\int_{\partial\O}u\,d\sigma\quad\mbox{on }\partial\O\cap\{u=0\}.
\end{cases}\]
Then, if the solution $u$ is positive on $\partial\O$, the normal derivative $\partial u/\partial\nu$ is constant along $\partial\O$. This constant can be easily computed integrating both sides of the PDE above, and we obtain
$$\frac{\partial u}{\partial\nu}=-\frac{1}{|\partial\O|}\int_\O f\,dx\;.$$
\end{oss}

Coming back to the optimization problem \eqref{pb1}, the optimal density $h$ can be recovered by determining the (unique in the case $\O$ connected) solution $\bar u$ of the auxiliary problem \eqref{pb1aux} and then taking
$$h_{opt}=m\frac{\bar u}{\int_{\partial\O}|\bar u|\,d\HH^{d-1}}$$
if $\bar u\notin H^1_0(\O)$. As a consequence of the uniqueness above, if $\O$ is a ball and $f$ is radial, the optimal density $h_{opt}$ is constant along $\partial\O$.

It is interesting to notice that, taking $tu$ instead of $u$ in \eqref{pb1aux}, and optimizing with respect to $t$, gives the equivalent formulation of the auxiliary problem \eqref{pb1aux}
\begin{equation}\label{pb1aux.1}
\min\left\{\frac{\int_\O|\nabla u|^2\,dx+\frac1m\Big(\int_{\partial\O}|u|\,d\HH^{d-1}\Big)^2}{\Big(\int_\O fu\,dx\Big)^2}\ :\ u\in H^1(\O)\right\}\;.
\end{equation}

%%%%%%%%%%%%%%%%%%%%%%%%%%%%%%
\section{The eigenvalue problem}\label{seigen}

In this section we consider the optimization problem for the first eigenvalue of the operator \eqref{operator}, that is the minimization problem \eqref{pb2}. Similarly to what done in the previous section, we can interchange the two minimizations, obtaining the auxiliary problem
\be\label{pb2aux}
\min\left\{\frac{\int_\O|\nabla u|^2\,dx+\frac1m\Big(\int_{\partial\O}|u|\,d\HH^{d-1}\Big)^2}{\int_\O u^2\,dx}\ :\ u\in H^1(\O)\right\}\;.
\ee
The existence of a solution $\bar u$ easily follows from the direct methods of the calculus of variations; we may also assume that $\bar u$ is nonnegative. Again, this gives the optimal density $h_{opt}$ by
$$h_{opt}=m\frac{\bar u}{\int_{\partial\O}\bar u\,d\HH^{d-1}}\;.$$
We want to investigate about the radial symmetry of $\bar u$ (hence on the fact that $h_{opt}$ is constant) in the case when $\O$ is a ball. The surprising symmetry breaking is contained in the following result.

\begin{teo}\label{nonsimm}
Let $\O$ be a ball. Then there exists $m_0>0$ such that the solution of the auxiliary variational problem \eqref{pb2aux} is radial if $m>m_0$, while the solution is not radial for $0<m<m_0$. As a consequence, the optimal density $h_{opt}$ is not constant if $m<m_0$.
\end{teo}

\begin{proof}
Set for every $m>0$
$$J_m(u)={\int_\O|\nabla u|^2\,dx+\frac1m\Big(\int_{\partial\O}|u|\,d\sigma\Big)^2
\over\int_\O u^2\,dx}\;,\qquad\lambda_m=\min\Big\{J_m(u)\ :\ u\in H^1(\O)\Big\}\;.$$
Moreover, let us denote by $\lambda_N$ the first nonzero eigenvalue of the Neumann problem:
$$\lambda_N=\min\Big\{J_\infty(u)\ :\ u\in H^1(\O),\ \int_\O u\,dx=0\Big\}$$
and by $\lambda_D$ the first eigenvalue of the Dirichlet problem
$$\lambda_D=\min\Big\{J_\infty(u)\ :\ u\in H^1_0(\O)\Big\}\;.$$
Observe that 
$\lambda_m$ is decreasing in $m$ and 
$$\lambda_m\to0\qquad\hbox{as }m\to\infty,$$
 while
$$\lambda_m\to\lambda_D\qquad\hbox{as }m\to0.$$
Therefore there exists a unique positive $m$ such that $\lambda_{m}=\lambda_N$ and we want to prove that when $\O=B_R$ such a value is indeed the threshold value $m_0$ in the statement.

For $\O=B_R$ and $m<m_0$ assume by contradiction that the solution $u$ to the auxiliary problem \eqref{pb2aux} is radial; then it is positive and we may take $u+\eps v$ as a test function, where $v$ is the first eigenfunction of the Neumann problem. Without loss of generality we may assume that $\int_\O u^2\,dx=\int_\O v^2\,dx=1$. Using the fact that $u$ and $v$ are orthogonal, and that $\int_{\partial\O}v\,d\sigma=0$ we have
$$\lambda_m=J_m(u)\le J_m(u+\eps v)={\lambda_m+\eps^2\lambda_N\over1+\eps^2}$$
which implies $\lambda_m\le\lambda_N$ in contradiction to $m<m_0$ and $\lambda_m>\lambda_N$. 

Therefore radial solution do not exist when $m<m_0$, and the proof is complete if we show that non radial solutions only exist if $m\le m_0$. 

First we observe that, when $m\ne m_0$, positive solutions to the auxiliary problem \eqref{pb2aux} exists if and only if they are radial. In fact, on one hand we know that any radial solution is necessarily positive. On the other hand a positive function $u$ is a solution to problem \eqref{pb2aux} if and only if
\be\label{pb2pde1}
\begin{cases}
-\Delta u=\lambda_m u\quad\hbox{in }\O,\\
\ds\frac{\partial u}{\partial\nu}=-\frac1m\int_{\partial\O}u\,d\sigma\quad\hbox{on }\partial\O.
\end{cases}
\ee

Let $u(r,\omega)$ be such a solutions in polar coordinate $(r,\omega)\in\mathbb R^+\times \mathbb{S}^{d-1}$. By averaging $u$ along the angular coordinate $\omega$ and using the linearity of problem \eqref{pb2pde1}, we obtain a radial function solution to \eqref{pb2pde1} and therefore a radial solution to \eqref{pb2aux}. Since any two positive solutions to \eqref{pb2pde1} can be linearly combined to obtain a solution to the Neumann eigenvalue problem, the fact that $\lambda_m\ne \lambda_N$ implies that positive solutions to \eqref{pb2pde1} are unique up to a multiplicative constant.

%Arguing again by contradiction 
We assume now that for some $m\ne m_0$ there exists a non radial positive solution $u$ to \eqref{pb2aux}. From what we have just observed we know that such a solution cannot be positive and has to vanish somewhere on the boundary of $\O$. Let $u(r,\omega)$ be such a solution in polar coordinates. By using the spherical symmetrization on \eqref{pb2aux} we can always assume that there exists $\omega_0\in\mathbb{S}^{d-1}$ such that $u(r,\omega)$ is spherically symmetric in the direction $\omega_0$, that is:
\be\label{monotonicity}
u(r,\omega_1)\ge u(r,\omega_2)\qquad\hbox{ for all $0<r<R$, whenever }|\omega_1-\omega_0|\le|\omega_2-\omega_0|.
\ee
We also know that $u$ is a solution to
\be\label{pb2pde2}
\begin{cases}
-\Delta u=\lambda_m u\quad\mbox{in }\O,\\
\ds\frac{\partial u}{\partial\nu}=-\frac1m\int_{\partial\O}u\,d\sigma\quad\mbox{on }\partial\O\cap\{u>0\},\\
\ds\frac{\partial u}{\partial\nu}\ge-\frac1m\int_{\partial\O}u\,d\sigma\quad\mbox{on }\partial\O\cap\{u=0\}.
\end{cases}\ee
We multiply the eigenvalue equation in \eqref{pb2pde2} by the first nontrivial eigenfunction $v$ of the Neumann eigenvalue problem on $\O$ and integration yields
\be\label{antisym}
(\lambda_m-\lambda_N)\int_\O uv\,dx=-\int_{\partial\O}\frac{\partial u}{\partial\nu}v\,d\sigma
\ee
We can choose $v$ so that it is spherically symmetric in the direction $\omega_0$. In such a case we have $\int_\O uv\,dx>0$; indeed $v$ is also antisymmetric with respect to reflection about the equatorial plane of $\O$ orthogonal to $\omega_0$ and therefore $\int_\O uv\,dx>0$ as long as $u$ is nonradially symmetric. Moreover, the right-hand side of \eqref{antisym} is nonnegative; to see this fact, we notice that from \eqref{pb2pde2} we have that for a suitable $\delta>0$
$$\begin{cases}
u(R,\omega)=0\qquad\hbox{if }|\omega-\omega_0|\ge\delta\\
u(R,\omega)>0\qquad\hbox{otherwise.}
\end{cases}$$
From \eqref{monotonicity} and from the boundary conditions in \eqref{pb2pde2} we deduce that $\partial u/\partial\nu$ is spherically symmetric in the direction $\omega_0$ (in the sense of \eqref{monotonicity}) and then we can repeat the argument above.

\noindent Eventually we deduce $\lambda_m\ge\lambda_N$ and $m\le m_0$.
\end{proof}

\begin{oss}
The symmetry breaking result of Theorem \ref{nonsimm} has the following physical interpretation: if we want to insulate a given circular domain in order to have the slowest decay of the temperature, the best insulation around the boundary has a constant thickness if we have enough insulating material at our disposal. On the contrary, if the total amount of insulator is small, the best distribution of it around the boundary is nonconstant.
\end{oss}

We show now that when the dimension $d$ is one no symmetry breaking occurs. Indeed, taking as $\O$ the interval $]-1,1[$ and setting $\lambda_m=\omega^2$, the solution $u$ of equation \eqref{pb2pde2} is of the form
$$u(x)=\cos(\omega x+\alpha)$$
with the boundary conditions given by \eqref{pb2pde2}
$$\begin{cases}
|u'(a)|=\big(u(-1)+u(1)\big)/m&\hbox{if }u(a)>0\\
|u'(a)|\ds\le\big(u(-1)+u(1)\big)/m&\hbox{if }u(a)=0
\end{cases}
\qquad\hbox{for }a=\pm1.$$
This easily implies that $\alpha=0$ that is the solution $u$ is symmetric. In this case $\omega$ turns out to be the unique solution of the equation
$$\tan\omega=\frac{2}{m\omega}\;.$$
We notice that in dimension $d=1$ the first nontrivial Neumann eigenvalue $\lambda_N(\O)$ coincides with the first Dirichlet eigenvalue $\lambda_D(\O)$, so one could also repeat the argument in the proof of Theorem \ref{nonsimm} to conclude again that $u$ has to be symmetric.

%%%%%%%%%%%%%%%%%%%%%%%%%%%%%%
\section{Further remarks}\label{sfurther}

We point out the two shape optimization problems that arise from \eqref{pb1} and \eqref{pb2} when we let $\O$ to vary too. Denoting by $E_\O(h)$, $\HH_{m,\O}$, $\lambda_{m,\O}$ the same mathematical objects introduced above, where the dependence on the domain $\O$ is stressed, we may consider the problem of determining the optimal domain $\O$, among the ones having a prescribed Lebesgue measure, to obtain the best insulation. Considering the two criteria above, the first shape optimization problem becomes
\be\label{shopt1}
\min\big\{E_\O(h)\ :\ h\in\HH_m,\ |\O|\le1\big\}.
\ee
As done in Section \ref{senergy} we can eliminate the variable $h$ and the minimization problem becomes
$$\min\left\{\frac{\int_\O|\nabla u|^2\,dx+\frac1m\Big(\int_{\partial\O}|u|\,d\HH^{d-1}\Big)^2}{\Big(\int_\O fu\,dx\Big)^2}\ :\ u\in H^1(\O),\ |\O|\le1\right\}\;.$$
A constraint $\O\subset D$ can be added, where $D$ is a given {\it bounded} domain.

Assume that $f\in L^\infty(D)$, $f\ge0$, $\O$ is Lipschitz and $u\ge0$ is a minimizer in \eqref{pb1aux.1}. Following the ideas developed in \cite{bugi10,bugi15sv}, we observe that the function $u$ extended by zero on $D\setminus\O$ belongs to $SBV(D)$. Consequently, one can relax the shape optimization problem, by replacing the couple $(\O,u)$ with a new unknown $v\in SBV(D,\R^+)$, the set $\O$ being identified with $\{v>0\}$. Then, the shape optimization problem becomes
\be\label{shopt1r}
\min\left\{\frac{\int_D|\nabla^a v|^2\,dx+\frac1m\Big(\int_{J_v}v^++v^-\,d\HH^{d-1}\Big)^2}{\Big(\int_D fv\,dx\Big)^2}\ :\ v\in SBV(D,\R^+),\ |\{v\ne0\}|\le1\right\}.
\ee
Above, $\nabla^a v$ denotes the absolute continuous part of the distributional gradient of $v$, with respect to the Lebesgue measure, $J_v$ denotes the jump set of $v$ and $v^+,v^-$ the upper and lower approximate limits of $v$ at a jump point. Problem \eqref{shopt1r} is indeed a relaxation of the original shape optimization problem, as a consequence of the density result \cite[Theorem 3.1]{coto99}.

Problem \eqref{shopt1r} has a solution, the proof being done by the direct method of the calculus of variations. The main ingredient is the $SBV$ compactness theorem adapted as in \cite[Theorem 2]{bugi10} to Robin boundary conditions, with the only difference that compactness occurs in $L^1(D)$. Using the hypothesis $f\in L^\infty(D)$, the existence of an optimal $SBV$ solution follows. The fact that the optimal $SBV$ solution corresponds to a "classical" solution needs a quite technical investigation and faces the difficulty that the norm of the traces at the jump set is controlled only in $L^1$.

\medskip
The case $f\equiv1$ is particular and requires special attention. In some sense, the optimization question is related to a torsion-like problem of Saint-Venant type. For Robin boundary conditions it was proved in \cite{bugi15sv} that the ball is a minimizer. One could reasonably expect that this result should hold as well for the shape optimization problem \eqref{shopt1}, in the case $f\equiv1$ and $D=\R^d$. By direct computation, one can notice that among all balls satisfying the measure constraint, the largest is the solution. Moreover, if $\O$ a union of disjoint balls of different radii, the solution $u$ is constant on the boundary of one ball and has to vanish on the boundary of all the others, as already seen in Example \ref{twoballs}. Then, using the classical Saint-Venant inequality, a direct computation leads to the optimality of one single ball.

\medskip

Similarly, the second shape optimization problem, arising from the minimization of the first eigenvalue, considered in Section \ref{seigen} is
$$\min\left\{\frac{\int_\O|\nabla u|^2\,dx+\frac1m\Big(\int_{\partial\O}|u|\,d\HH^{d-1}\Big)^2}{\int_\O u^2\,dx}\ :\ u\in H^1(\O),\ |\O|\le1\right\}\;.$$
Let us notice that the shape optimization problem above, does not have a solution, as soon as the dimension $d$ of the space is larger than $2$. Indeed, when $d\ge3$ we may consider the domains $\O_n=B_{1/n}$, the ball of radius $1/n$, and take the test function $u_n=1$. Then
$$\lambda_m(\O_n)\le\frac1m\frac{\Big(\HH^{d-1}(\partial B_{1/n})\Big)^2}{|B_{1/n}|}=\frac{d^2\omega_d}{mn^{d-2}}$$
and, as $n\to\infty$, we get $\lambda_m(\O_n)\to0$. Similarly, if $|\O|=1$ and $\O$ is the union of $n$ disjoint balls of volume $1/n$ each, we obtain
$$\lambda_m(\O)\le\lambda_m(B_{1/n})\to0.$$

\medskip

It would be interesting to prove that for the two problems above (with $d=2$ for the second one) an optimal shape $\O_{opt}$ exists, even if in a rather large class of domains with very mild regularity. Also, in the case of the energy problem with $f=1$ (often called {\it torsion} problem) it would be interesting to prove (or disprove) that the optimal domain is a ball.

Nevertheless, in the proposition below we prove that for $m<m_0$ the ball cannot be a stationary domain for the functional $\lambda_m(\O)$.

\begin{teo}\label{nonstat}
Let $m<m_0$; then the ball cannot be a stationary domain for the functional $\lambda_m(\O)$, where stationarity is intended with repect to smooth perturbations of the boundary, that is for every smooth and compactly supported field $V(x)$ with $\dive V=0$, setting $\O_\eps=(Id+\eps V)(\O)$, we have
$$\lambda_m(\O_\eps)-\lambda_m(\O)=o(\eps).$$
\end{teo}

\begin{proof}
Let us take the ball of unitary radius, assume by contradiction that it is stationary with respect to all smooth perturbation of the boundary that preserve the Lebesgue measure, and set
$$\Gamma=\{x\in\partial\O\ :\ u(x)=0\}.$$
Up to a spherical symmetrization, the set $\Gamma$ above is a nonvanishing spherical cap (an arc of circle in dimension two); indeed if it reduces to a point (or to the empty set), the normal derivative $\partial u/\partial\nu$ is constant on $\partial\O$ except at most a point; averaging with respect to $\omega$ the function $u(r,\omega)$ provides then a symmetric positive first eigenfunction. Thanks to Theorem \ref{nonsimm} this implies $m\ge m_0$ which is not true.
Set now $\O_\eps=\{x+\eps V(x)\ :\ x\in\O\}$ with $\dive V=0$ and $V$ orthogonal to $\partial\O$, and let $u\ge0$ be a first eigenfunction of problem \eqref{pb2aux}. Setting
$$u_\eps(x)=u\big((Id+\eps V)^{-1}(x)\big),$$
by standard computations (see for instance \cite{hepi05}) and assuming $u$ is smooth enough up to the boundary (which is true in our case since $\O$ is the ball) we obtain
$$\lambda_m(\O_\eps)\le J_m(u_\eps)=\lambda_m(\O)+\eps\int_{\partial\O}j_m(u)\,V\cdot\nu\,d\HH^{d-1}+o(\eps),$$
where
$$j_m(u)=|\nabla_\tau u|^2-|\nabla_\nu u|^2-\lambda_m u^2+\frac2m\Big(\int_{\partial\O}u\Big)H(x)u\;,$$
being $H$ the mean curvature in $\O$ (in our case constant, since $\O$ is the ball). By the stationarity of $\O$ and since
$$\int_{\partial\O}V\cdot\nu\,d\HH^{d-1}=0$$
we get
$$|\nabla_\tau u|^2-|\nabla_\nu u|^2-\lambda_m u^2+\frac2m\Big(\int_{\partial\O}u\Big)H(x)u=const\;.$$
In particular, this gives, for a suitable constant $c$
$$\frac{\partial u}{\partial\nu}=c\qquad\hbox{on the set }\Gamma=\{x\in\partial\O\ :\ u(x)=0\}.$$
The conclusion that $u$ has to be radial now follows by the Holmgren uniqueness theorem. Indeed, the radial solution, given by the ODE
$$-w''(r)+\frac{d-1}{r}w'(r)=\lambda_m w(r),\quad w(1)=0,\quad w'(1)=c,$$
also satisfies the PDE
$$\begin{cases}
-\Delta u=\lambda_m u\\ u=0\hbox{ on }\Gamma,\quad\frac{\partial u}{\partial\nu}=c\hbox{ on }\Gamma
\end{cases}$$
which gives that $u(x)=w(|x|)$ in a neighborhood of $\partial\O$, which is impossible by Theorem \ref{nonsimm} since we assumed $m<m_0$.
%by the same argument used in \cite{heou} (see also \cite{frga08}), that we outline here for a sake of completeness. Take
\end{proof}

As said above, in the case of the energy we are not able to prove that the ball solves the shape optimization problem \eqref{shopt1}; nevertheless, we can say it is a stationary domain, in the same sense of Theorem \ref{nonstat}, that is
$$E(\O_\eps)-E(\O)=o(\eps)$$
for every $\O_\eps=(Id+\eps V)(\O)$, where the quantity $E(\O)$ is defined as
$$E(\O)=\min\left\{\frac12\int_\O|\nabla u|^2\,dx+\frac{1}{2m}\Big(\int_{\partial\O}|u|\,d\HH^{d-1}\Big)^2-\int_\O u\,dx\ :\ u\in H^1(\O)\right\}.$$
Indeed, the solution $u$ on the ball is positive and radially symmetric by the uniqueness result of Proposition \ref{unique}. Therefore, the solutions $u_\eps$ relative to $\O_\eps$ are still positive and solve a Neumann problem of the form
$$\begin{cases}
-\Delta u_\eps=1\hbox{ in }\O_\eps\\
\ds\frac{\partial u_\eps}{\partial\nu}=-\frac1m\int_{\partial\O_\eps}u_\eps\,d\HH^{d-1}.
\end{cases}$$
The same argument used in the proof of Theorem \ref{nonstat} provides the stationarity condition
$$|\nabla_\tau u|^2-|\nabla_\nu u|^2-u+\frac2m\Big(\int_{\partial\O}u\Big)H(x)u=const$$
which is fulfilled in our case since we know that $u$ is radial.

\bigskip\ack The work of the first author is part of the ANR Optiform research project, ANR-12-BS01-0007. The work of the second author is part of the project 2010A2TFX2 {\it``Calcolo delle Variazioni''} funded by the Italian Ministry of Research and University. The second and third authors are members of the {\it``Gruppo Nazionale per l'Analisi Matematica, la Probabilit\`a e le loro Applicazioni''} (GNAMPA) of the {\it``Istituto Nazionale di Alta Matematica''} (INDAM).

%%%%%%%%%%%%%%%%%%%%%%%%%%%%%%

\bigskip
{\small\noindent
Dorin Bucur:
Laboratoire de Math\'ematiques (LAMA),
Universit\'e de Savoie\\
Campus Scientifique,
73376 Le-Bourget-Du-Lac - FRANCE\\
{\tt dorin.bucur@univ-savoie.fr}\\
{\tt http://www.lama.univ-savoie.fr/$\sim$bucur/}

\bigskip\noindent
Giuseppe Buttazzo:
Dipartimento di Matematica,
Universit\`a di Pisa\\
Largo B. Pontecorvo 5,
56127 Pisa - ITALY\\
{\tt buttazzo@dm.unipi.it}\\
{\tt http://www.dm.unipi.it/pages/buttazzo/}

\bigskip\noindent
Carlo Nitsch:
Dipartimento di Matematica e Applicazioni,
Universit\`a di Napoli ``Federico II''\\
Via Cintia, Monte S. Angelo,
80126 Napoli - ITALY\\
{\tt carlo.nitsch@unina.it}\\
{\tt http://wpage.unina.it/c.nitsch/}

\end{document}